\documentclass[a4paper,10pt]{amsart}
\usepackage[utf8]{inputenc}
\usepackage{lmodern}
\usepackage[english]{babel}
\usepackage[T1]{fontenc}
\usepackage{microtype} 
    \usepackage[OT2,OT1]{fontenc}
    \newcommand\cyr{%
    \renewcommand\rmdefault{wncyr}%
    \renewcommand\sfdefault{wncyss}%
    \renewcommand\encodingdefault{OT2}%
    \normalfont
    \selectfont}
    \DeclareTextFontCommand{\textcyr}{\cyr}
\usepackage{amsmath,amssymb,amsfonts,amsthm,amscd,latexsym,mathrsfs}
\usepackage[all]{xy}
\input xypic
\usepackage{color}
\usepackage{hyperref}
\usepackage{lipsum}
\usepackage{breqn}
\hypersetup{colorlinks=true,linkcolor=red,citecolor=blue}
\usepackage[hyperpageref]{backref} 
\usepackage{booktabs}
\usepackage{multirow}
\theoremstyle{plain}
\newtheorem{theorem}[subsection]{{\bf Theorem}}
\newtheorem*{theorem*}{{\bf Theorem}}
\newtheorem{corollary}[subsection]{{\bf Corollary}}
\newtheorem*{corollary*}{{\bf Corollary}}

\newtheorem{lemma}[subsection]{{\bf Lemma}}

\newtheorem{question}[subsection]{{\bf Question}}

\theoremstyle{definition}
\newtheorem*{definition}{{\bf Definition}}
\theoremstyle{remark}

\numberwithin{equation}{section}
\DeclareMathOperator{\Res}{Res}
\DeclareMathOperator{\Inf}{Inf}
\DeclareMathOperator{\Tra}{Tra}
\DeclareMathOperator{\inff}{inff}
\DeclareMathOperator{\tra}{tra}
\DeclareMathOperator{\Hom}{Hom}

\newcommand{\gen}[1]{\big{\langle} #1 \big{\rangle}}

\newcommand{\largewedge}{\mbox{\Large $\wedge$}}
\begin{document}
\baselineskip=14pt
\title{Bogomolov multiplier of Lie Algebras}
%%%%%%%%%%%%%%%%%%%%%%%%%%%%%%%%%%%%%%%%%%%%%%%%%%%%%
%\author[B. Kunyavski\u{\i}]{Boris Kunyavski\u{\i}}
%\address[Boris Kunyavski\u{\i}]{
%Department of Mathematics, Bar-Ilan University
%Ramat Gan \\
%Israel}
%\email{kunyav@gmail.com}
%%%%%%%%%%%%%%%%%%%%%%%%%%%%%%%%%%%%%%%%%%%%%%%%%%%%%
\author[P. K. Rai]{Pradeep K. Rai}
\address[Pradeep K. Rai]{Mahindra University, 
Hyderabad, Telangana,\\
India}
\email{raipradeepiitb@gmail.com}
%%%%%%%%%%%%%%%%%%%%%%%%%%%%%%%%%%%%%%%%%%%%%%%%%%%%%
%%%%%%%%%%%%%%%%%%%%%%%%%%%%%%%%%%%%%%%%%%%%%%%%%%%%%
\begin{abstract}
In the work of Rostami et al., the Bogomolov multiplier of a Lie algebra $L$ over a field $\Omega$ is defined as a particular factor of a subalgebra of the exterior product $L \wedge L$. If $L$ is finite dimensional, we identify this object as a certain subgroup of the second cohomology group $H^2(L, \Omega)$ by deriving a Hopf-Type formula. As an application, we affirmatively answer two questions posed by Kunyavski\u{\i} regarding the invariance of the Bogomolov multiplier under isoclinism of Lie algebras and the existence of a family of Lie algebras with  Bogomolov multipliers of unbounded dimension.

\end{abstract}
%%%%%%%%%%%%%%%%%%%%%%%%%%%%%%%%%%%%%%%%%%%%%%%%%%%%%
\maketitle

\section{Introduction}
The Bogomolov multiplier of a finite group $G$ is a cohomological invariant of $G$ that has been studied in connection with Noether's problem, which asks whether the fixed subfield $k(G)$ of the function field $k(x_g: g \in G)$ is purely transcendental over an algebraically closed field $k$ of characteristic zero. The Bogomolov multiplier has played a key role in the discovery of counter examples to Noether's problem over the complex numbers $\mathbb{C}$. Saltman \cite{Saltman} was the first to provide such counter examples, showing that if the unramified cohomology group $H_{nr}^2(k(G), \mathbb{C})$ is nonzero, then $G$ has a negative solution to Noether's problem over $\mathbb{C}$. Bogomolov later showed that $H_{nr}^2(k(G), \mathbb{C})$ is naturally isomorphic to the subgroup $B_0(G)$ of the second cohomology group $H^2(G,\mathbb{C})$ consisting of classes that vanish when restricted to the abelian subgroups of $G$ \cite{Bogomolov}. This has led to the discovery of numerous other counter examples to Noether's problem. Kunyavski\u{\i} later gave the name ``Bogomolov multiplier" to $B_0(G)$ \cite{Kunyavskii}, and Moravec provided a homological description of $B_0(G)$ as a quotient of $H^2(G,\mathbb{C})$ \cite{Moravec}. Following Moravec's construction, Rostami et. al \cite{Niroomand} extended the notion of Bogomolov multiplier to Lie algebras. For the convenience of the reader we define it here.

Let $L$ be a Lie algebra over $\Omega$. The exterior square of $L$ 
is defined to be the Lie algebra $L \largewedge L$ generated by the symbols $m \wedge n$, where $m, n \in L$, subject to the following relations:
\begin{itemize}
    \item[(i)] $\alpha(m \wedge n) = \alpha m \wedge n = m \wedge \alpha n$,
\item[(ii)] $(m + m') \wedge n = m \wedge n + m' \wedge n$,
\item[(iii)] $m \wedge (n + n') = m \wedge n + m \wedge n'$,
\item[(iv)] $[m, m'] \wedge n = m \wedge [m', n]- m' \wedge [m, n]$,
\item[(v)] $m \wedge [n, n'] = [n', m] \wedge n -[n, m] \wedge n'$,
\item[(vi)] $[(m \wedge n),(m' \wedge n')] = -[n, m] \wedge [m', n']$,
\item[(vii)] $m \wedge n = 0$ whenever $m = n$,
\end{itemize}
for all $\alpha \in \Omega, m, m', n, n' \in L$.

It is easy to see that $\mathcal{K}: L \times L \mapsto [L, L]$ given by $(m, n) \mapsto [m, n]$ induces a homomorphism $\bar{\mathcal{K}}: L \largewedge L \mapsto [L, L]$, such that
$\bar{\mathcal{K}}(m \wedge n) = [m, n]$, for all $m, n \in L$. It is known that the kernel of $\bar{\mathcal{K}}$ is isomorphic to the Schur Multiplier $H^2(L,\Omega)$ (defined below). We denote it by $\mathcal{M}(L)$. Define $\mathcal{M}_0(L)$ to be the group $\gen{m \wedge n \ \mid \  m, n \in L, [m, n] = 0}$. %In Case $M= N = L$, $\mathcal{M}(M, N)$ and $\mathcal{M}_0(M,N)$ are denoted by $\mathcal{M}(L)$ and $\mathcal{M}_0(L)$ respectively.\\
The factor group $\frac{\mathcal{M}(L)}{\mathcal{M}_0(L)}$ is defined to be the Bogomolov multiplier $B(L)$ of the Lie algebra $L$. 

In this article, we define a cohomological object $B_0(L)$ for a finite dimensional Lie algebra $L$ over a field $\Omega$, and show that it is isomorphic to the Bogomolov multiplier $B(L)$. Before that, we recall the definition of the Schur multiplier of a finite dimensional Lie algebra. Let $L$ be a finite dimensional Lie algebra and $A$ be a trivial $L$-module. A map $f: L \times L \mapsto A$ said to be a 2-cocycle if it is bilinear, alternating and satisfies 
\[f([x_1, x_2], x_3) + f([x_2, x_3], x_1) + f([x_3, x_1], x_2) = 0.\]
And $f$  is  said to be a 2-coboundry if there exists a linear $\sigma : L \mapsto A$ such that 
\[f(x_1, x_2) = -\sigma([x_1, x_2]).\]
The sets of 2-cocycles and 2-coboundries are denoted by $Z^2(L,A)$ and  $B^2(L,A)$, respectively and form abelian groups with respect to usual addition. The group $Z^2(L,A)/B^2(L,A)$ is said to be the second cohomology group with coefficients in $A$, and is denoted by $H^2(L,A)$. 

Schur multiplier of the Lie algebra $L$ is defined as the abelian Lie algebra $H^2(L,\Omega)$, considering $\Omega$ as a central $L$-module. We are now ready to define $B_0(L).$

%Next, we define $B_0(L)$ to be the set of those cohomology classes which vanishes on abelian Lie subalgebras. The following Lemma is easy to observe:
\begin{definition}
For a finite dimensional Lie algebra $L$ over $\Omega$, we define $B_0(L)$ as follows:
\[B_0(L) = \{\overline{f} \in H^2(L,\Omega)\ \ | \ \ f(x_1, x_2) = 0 \ \ \text{whenever} \ \ [x_1, x_2] = 0\}.\]
\end{definition}

Batten \cite[Theorem 3.6]{Batten} established the following Hopf Formula for the Schur multiplier of the Lie algebra $L$:  
\[H^2(L,\Omega) \cong \frac{F' \cap R}{[F,R]}, \]
where $1 \mapsto R \mapsto F \mapsto L \mapsto 1$ is a free a presentation of $L$ and $F'$ is the  derived subalgebra of $F$. 

Let $K(L)$ denote the set $\{[x,y] \ \ | \ \ x, y \in L\}$.
In the following theorem we derive a Hopf-type formula for $B_0(L)$. 

\begin{theorem}\label{cor}
Let $L$ be a finite dimensional Lie algebra with a free presentation $L \cong F/R$. Then $B_0(L) \cong \frac{F' \cap R}{\gen{K(F) \cap R}}$.
\end{theorem}

The following corollary follows from \cite[Proposition 4.1]{Niroomand} and the Theorem \ref{cor}.
\begin{corollary}\label{cor1}
Let $L$ be a finite dimensional Lie algebra. Then $B(L) \cong B_0(L).$
\end{corollary}

As an application we answer a couple of questions of Kunyavski\u{\i} \cite{Kunyavski}. He asked the following questions for finite dimensional Lie algebras $L$:

\begin{question}\label{q1}\cite[Question 7.1]{Kunyavski}\label{q1}
Can the dimension of $B(L)$ be as large as possible?
\end{question}

%Before stating the next question of Kunyavski we need to define the isoclinism of Lie algebras. 
\begin{question}\label{q2}\cite[Question 7.2]{Kunyavski}\label{q2}
Is $B(L)$ invariant under isoclinism of Lie algebras?
\end{question}
Two Lie algebras $L$ and $K$ are said to be isoclinic if there exist isomorphisms $\alpha: L/Z(L) \mapsto K/Z(K)$ and $\beta : L' \mapsto K'$ such that the following diagram commutes. %if α(l1 + Z(L)) = k1 + Z(K) and α(l2 + Z(L)) = k2 + Z(K)
%then β([l1, l2]) = [k1, k2]. 
\begin{center}
                        $ \begin{CD}
\frac{L}{Z(L)} \times \frac{L}{Z(L)}        @>\phi >>  L'    \\
@V \alpha \times \alpha VV             @V \beta VV\\
\frac{K}{Z(K)}\times \frac{K}{Z(K)}          @> \theta >>     K'\\  
\end{CD}$
\end{center}
 %\[\mathcal{P}(\overline{x}, y) = (\overline{x}, \overline{y}) \ \ \text{for} \ x \in G, y \in \gamma_2(G),\]
where \[\phi\Big(l_1 +Z(L), l_2+Z(L)\Big) = [l_1,l_2] \ \ \text{for}\ \  l_1,l_2 \in L\] and
\[\theta\Big(k_1 +Z(K), k_2+Z(K)\Big) = [k_1,k_2] \ \ \text{for} \ \ k_1, k_2 \in K.\] 
The pair $(\alpha, \beta)$ is called an isoclinism between $L$ and $K$.\\

In the following theorems we give an affirmative answer to Questions \ref{q1} and \ref{q2}.

\begin{theorem}\label{thm1}
Let $n \geq 1$ be a natural number. There exists a finite dimensional nilpotent Lie algebra $L$ of nilpotency class 2 such that dimension of $B(L)$ is greater than or equal to $n$. 
\end{theorem}

\begin{theorem}\label{thm2}
Let $L$ and $M$ be two isoclinic finite dimensional Lie Algebras over the field $\Omega$. Then $B(L) \cong B(M)$.
\end{theorem}

%To prove these theorems we first prove that the Bogomolov multiplier $B(L)$ of a finite dimensional Lie algebra $L$ (Over the field $\Omega$) is isomorphic to a certain subgroup $B_0(L)$ (defined later in the section)  of the second cohomology group of $L$ with coefficients in $\Omega.$

\section{Hopf-Type Formula}
%Let $L$ be a finite dimensional Lie algebra over the field $\Omega$, $H$ be a central ideal of $L$ and $A$ be a trivial $L$-module. To any homomorphism $f : L \rightarrow A $ we can assign its
%restriction to $H$, thus obtaining a homomorphism
%$\Hom(L, A) \rightarrow \Hom(H, A).$ This homomorphism is said to be restriction map and is denoted by $\Res.$ Suppose now that $\beta: L \rightarrow L/H$ is the natural group homomorphism. Then clearly $\beta$ induces a homomorphism $\Hom(L/H, A) \rightarrow \Hom(L, A)$. An induced homomorphism of second cohomology groups can be defined as follows. Given $\alpha \in Z^2(L, A)$, we can assign to $\alpha$ the cocycle $\alpha' \in Z^2(L/H, A)$ defined by
%$\alpha'(x, y) = \alpha(\beta(x), \beta(y)), \forall x, y \in L$. Then the assignment $\alpha \rightarrow \alpha'$
%induces a homomorphism $H^2(L/H, A) \rightarrow H^2(L, A)$. The induced homomorphisms $\Hom(L/H, A) \rightarrow \Hom(L, A)$ and $H^2(L/H, A) \rightarrow H^2(L, A)$ are called inflation maps and shall be denoted by $\Inf$. Next, choose a section $\mu$ of $\beta$ and define $f : L/H \times L/H \rightarrow L$ by $f(x, y) = [\mu(x), \mu(y)] - \mu([x,y])$ for all $x, y \in L/H$. Given $\chi \in \Hom(H, A)$, it is straightforward to verify that $\chi f \in Z^2(L/H, A)$ and that the cohomology class of $\chi f$ does not depend upon the choice of $\mu$. We can therefore define a map $\Hom(H, A) \rightarrow H^2(L/H, A)$ by mapping $\chi$ to $\chi f$ which is easily seen to be a homomorphism. This homomorphism is  called the transgression map and will be denoted by $\Tra.$

Consider a finite dimensional Lie algebra $L$ over a field $\Omega$, a central ideal $H$ of $L$, and a trivial $L$-module $A$. The restriction map $\Res: \Hom(L, A) \rightarrow \Hom(H, A)$ is defined as follows: for a homomorphism $f: L \rightarrow A$, $\Res(f)$ is the restriction of $f$ to $H$. There is also an inflation map $\Inf: \Hom(L/H, A) \rightarrow \Hom(L, A)$ defined by sending a homomorphism $\alpha \in \Hom(L/H, A)$ to the homomorphism $\alpha' \in \Hom(L, A)$ defined by $\alpha'(x, y) = \alpha(\beta(x), \beta(y))$ for all $x, y \in L$, where $\beta: L \rightarrow L/H$ is the natural group homomorphism. Another inflation map $\Inf: H^2(L/H, A) \rightarrow H^2(L, A)$ is defined similarly by sending $[\alpha] \in H^2(L/H, A)$ to $[\alpha'] \in H^2(L, A)$ where $\alpha'(x, y) = \alpha(\beta(x), \beta(y))$ for all $x, y \in L$.  Next, we define a transgression map $\Tra: \Hom(H, A) \rightarrow H^2(L/H, A)$ as follows: for a fixed section $\mu$ of $\beta$, define a map $f: L/H \times L/H \rightarrow L$ by $f(x, y) = [\mu(x), \mu(y)] - \mu([x,y])$ for all $x, y \in L/H$. Given a homomorphism $\chi \in \Hom(H, A)$, we can verify that $\chi f \in Z^2(L/H, A)$ and that the cohomology class of $\chi f$ does not depend on the choice of $\mu$. The transgression map $\Tra$ is then defined as the map that sends a homomorphism $\chi$ to the cohomology class of $\chi f$

We are now ready to quote some results required for our subsequent investigations. The following 5-term exact sequence was established by P. Batten in her Ph.D. Thesis \cite[Theorem 3.1]{Batten}
\begin{theorem} \label{thm6} Let $L$ be a Lie algebra, $H$ be  central ideal of $L$, $1 \mapsto H \mapsto L \mapsto L/H \mapsto 1$ be the natural exact sequence and $A$ be a trivial $L$-module. Then the induced sequence 
\begin{multline}\label{seq1}
1 \xrightarrow{} \Hom(L/H, A) \xrightarrow{\Inf} \Hom(L, A) \xrightarrow{\Res} \Hom(H, A) \xrightarrow{\Tra} H^2(L/H, A)\\
 \xrightarrow{\Inf} H^2(L, A)
\end{multline}
is exact.
\end{theorem}

Let $L$ be a lie algebra and $T$ be a subset of $L$. By $\gen{T}$ we denote the subspace of $L$ generated by $T$ and by $\Hom_{T}(L, A)$ we denote the set of those homomorphisms which maps $T$ to 0. The following lemma is instrumental in our subsequent investigations.

\begin{lemma}\label{lem0}
 Let $L$ be a finite dimensional Lie algebra over the field $\Omega$ and $H$ be a central ideal of $L$.  Then $\Tra(\lambda) \in B_0(L/H)$ if, and only if $\lambda \in \Hom_{T}(H, \Omega)$ where $T = \gen{K(L) \cap H}$.
\end{lemma}

\begin{proof}
 Let $\mu: L/H \rightarrow L$ be a section such that $\mu(0) = 0$ and let $\Tra$ be defined using $\mu$ as $\Tra(\lambda) = [\lambda f]$, where \[ f(x, y) = [\mu(x), \mu(y)] - \mu([x,y]) \ \ \forall \ x, y \in L/H.\]
Let $\lambda \in  \Hom_{T}(H, \Omega)$ and $x, y \in L/H$ be such that $[x,y] = 0.$ Then $\lambda f(x, y) = \lambda([\mu(x), \mu(y)]).$ But $[\mu(x), \mu(y)] \in T.$
Therefore $\lambda f(x,y) = 0$. This proves that $[\lambda f] \in B_0(L/H).$ Thus $\Tra(\lambda) \in B_0(L/H)$. Conversly, suppose that $\Tra(\lambda) \in B_0(L/H)$. Let $l_1, l_2 \in L$ such that $[l_1, l_2] \in H.$ Notice that $[l_1, l_2] = [\mu(l_1+H), \mu(l_2+H)]$ because $H$ is central. Also $\mu([l_1+H, l_2+H]) = \mu([l_1, l_2]+H) = 0$ because $[l_1, l_2] \in H$ and $\mu(0)=0.$ Therefore,
\[\lambda([l_1,l_2] = \lambda([\mu(l_1+H), \mu(l_2+H)])-\lambda\mu([l_1+H, l_2+H]) = \lambda f(l_1+H, l_2+H).\]
Since $\Tra(\lambda) = [\lambda f] \in B_0(L/H)$ and $[l_1+H, l_2+H] = 0$ in $L/H$ we have that $ \lambda f(l_1+H, l_2+H) = 0$. Thus $\lambda([l_1, l_2]) = 0$ whenever $[l_1, l_2] \in H.$ This proves that  $\lambda \in \Hom_{T}(H, \Omega).$   
\end{proof}

\begin{theorem} \label{thm7}  Let $L$ be a finite dimensional Lie algebra over the field $\Omega$, $H$ be  central ideal of $L$, and $T = \gen{K(L) \cap H}$. Then the induced sequence 
\begin{multline}
1 \xrightarrow{} \Hom(L/H, \Omega) \xrightarrow{\Inf} \Hom(L, \Omega) \xrightarrow{\Res} \Hom_{T}(H, \Omega) \xrightarrow{\tra} B_0(L/H) \\ \xrightarrow{\inff} B_0(L)
\end{multline}
is exact, where $\tra$ and $\inff$ are the restrictions of $\Tra$
 and $\Inf$. 
 \end{theorem}

\begin{proof}
The theorem follows from the exactness of the Sequence \ref{seq1}, Lemma \ref{lem0} and the following straight forward observations:

\begin{itemize}
    \item[(i)]  $\Res(\Hom(L, \Omega) \leq \Hom_{T}(H, \Omega)$.
    \item[(ii)] $\Inf(B_0(L/H) \leq B_0(L).$
   % \item[(iii)] $\Tra(\lambda) \in B_0(L/H)$ if, and only if $\lambda \in \Hom_{K(L) \cap H}(H, \Omega)$. 
\end{itemize}
\end{proof}

The next theorem follows from Lemma \ref{lem0} and \cite[Corollary 3.7]{Batten}.

\begin{theorem}\label{thm8}
Let $L$ be a finite dimensional Lie algebra, $L^*$ be its cover with $A \leq L^*$ satisfying the following three conditions 
\begin{enumerate}
\item $A \leq Z(L^*) \cap [L^*, L^*]$;
\item $A \cong M(L)$;
\item $L \cong L^*/A,$
\end{enumerate}
and $T = \gen{K(L^*) \cap A}$. Then the map $\tra:  \Hom_{T}(A, \Omega) \mapsto B_0(L)$ is bijective.
\end{theorem}

%\begin{proof}

%By \cite[Corollary 3.7]{Batten}, it is obvious that $\tra$ is injevtive. To see that it is surjective, apply \cite[Corollary 3.7]{Batten} again and note that $Tra(f) \in B_0(L)$ if, and only if $f \in Hom_{K(L^*) \cap A}(A, \Omega)$.

%\end{proof}

%If $T$ is a subset of an abelian Lie algebra $L$, then $\gen{T}$ denotes the subspace generated by $T.$

\begin{corollary}
Let $L$ be a finite dimensional Lie algebra and $L^*$ be its cover with $A \leq L^*$ satisfying the following three conditions 
\begin{enumerate}
\item $A \leq Z(L^*) \cap [L^*, L^*]$;
\item $A \cong M(L)$;
\item $L \cong L^*/A$.
\end{enumerate}
Then $B_0(L) \cong \frac{A}{\gen{K(L^*) \cap A}}$.
\end{corollary}

\begin{proof}
Let $T = \gen{K(L^*) \cap A}$. Since $A$ is an abelian lie algebra, homomorphisms are nothing but linear transformations. It follows that  $\Hom_{T}(A, \Omega) \cong \Hom\Big{(}\frac{A}{T}, \Omega\Big{)} \cong \frac{A}{T}$. The result follows from Theorem \ref{thm8}.
\end{proof}

\begin{theorem}\label{thm9}
Let $H$ be a central ideal in $L$ and $T = \gen{K(L) \cap H}$. Then $\frac{L' \cap H}{\gen{K(L) \cap H}}$ is isomorphic to the image of the map $\tra: \Hom_{T}(H, \Omega) \mapsto B_0(L/H)$. In particular,
%\begin{enumerate}
%\item 
%$\frac{L' \cap Z}{\gen{K(L) \cap Z}}$ is isomrphic to a subgroup of $M_E(G/Z)$,
%\item
$\frac{L' \cap H}{\gen{K(L) \cap H}} \cong B_0(L/H)$ if the map $\tra$ is surjective.
%\end{enumerate}
\end{theorem}

\begin{proof}

By Theorem \ref{thm7}, we have the following exact sequence 
\[\Hom(L, \Omega) \xrightarrow{\Res} \Hom_{T}(H, \Omega) \xrightarrow{\tra} B_0(G/H).\]
%Let $M$ be the subgroup of $Hom_{\gen{K(L) \cap Z}}(Z, \Omega)$ consisting of all $\chi$ which can be extended to a homomorphism $L \mapsto \Omega$. 
It follows that $\frac{\Hom_{T}(H, \Omega)}{\Res(\Hom(L,\Omega))}$ is isomorphic to the image of $\tra$. Thus, to prove the theorem, we only need to show that 
\[\frac{L' \cap H}{\gen{K(L) \cap H}} \cong \frac{\Hom_{T}(H, \Omega)}{\Res(\Hom(L,\Omega))}.\] 
Since $H$ is abelian it follows that %$\Res: Hom(Z, \Omega) \mapsto Hom(L' \cap H, \Omega)$ is surjective. Therfore 
the natural restriction map $\Res_1: \Hom_{T}(H, \Omega) \rightarrow \Hom_{T}(L' \cap Z, \Omega)$ is surjetive. Therefore
\[\frac{\Hom_{T}(H, \Omega)}{\ker \Res_1} \cong \Hom_T(L' \cap H, \Omega).\]

If we consider the natural restriction map $\Res_2: \Hom(H, \Omega) \mapsto \Hom(L' \cap H, \Omega)$, it is straight forward to note that $\ker \Res_1 = \ker \Res_2.$ Let $J$ be the subset of $\Hom(H, \Omega)$ consisting of all $\chi$ which can be extended to a homomorphism $L \rightarrow \Omega$. Invoking the proof of \cite[Theorem 3.2]{Batten} we have $J = \ker \Res_2.$ But it is obvious that $J = \Res(\Hom(L, \Omega))$. Hence, it follows that 

\[\frac{\Hom_{T}(H, \Omega)}{\Res(\Hom(L, \Omega))} \cong \Hom_T(L' \cap H, \Omega).\]
But 
\[\Hom_T(L' \cap H, \Omega) \cong \Hom\bigg(\frac{L' \cap H}{T}, \Omega\bigg) \cong \frac{L' \cap H}{T}\]
because $L' \cap H$ is abelian. This completes the proof.
\end{proof}

{\it {\bf Proof of Theorem \ref{cor}}}: Let $\bar{R} = \frac{R}{[F,R]}$ and $\bar{F} = \frac{F}{[F,R]}$.  Then $\bar{R}$ is a central ideal of $\bar{F}$ and $L \cong \frac{\bar{F}}{\bar{R}}$.  By \cite[Lemma 3.4]{Batten} the transgression map $\Tra: Hom(\bar{R}, \Omega) \mapsto H^2(L, \Omega)$ is surjective. Therefore by Lemma \ref{lem0} the map $\tra:  \Hom_{\gen{K(\bar{F}) \cap \bar{R}}}(\bar{R}, \Omega) \mapsto B_0(L)$ is also surjective. It therefore follows, from Theorem \ref{thm9}, that $B_0(L) \cong \frac{\bar{R} \cap [\bar{F}, \bar{F}]}{\gen{K(\bar{F}) \cap \bar{R}}}$. But $\bar{R} \cap [\bar{F}, \bar{F}] = \frac{F' \cap R}{[F,R]}$ and $\gen{K(\bar{F}) \cap \bar{R}} = \frac{\gen{K(F) \cap R}}{[F,R]}$. Hence $B_0(L) \cong \frac{F' \cap R}{\gen{K(F) \cap R}}$.

\vspace{.2cm}

\section{Applications}
The proof of the following proposition is exactly the same as the proof of \cite[Proposition 4.3]{Niroomand}.

\begin{theorem}\label{thm0}
Let $L$ be a Lie algebra with a free presentation $L \cong F/R$, and $M$ be an ideal of $L$, such that $T = \ker(F \mapsto L/M)$. Then the sequence
\[0 \mapsto \frac{R \cap \gen{K(F) \cap T}}{\gen{K(F) \cap R}} \mapsto B_0(L) \mapsto B_0(L/M) \mapsto \frac{M \cap L'}{\gen{K(L)\cap M}} \mapsto 0,\]
is exact.

\end{theorem}

\begin{definition}
A Lie algebra $L$ is called generalized Heisenberg of rank $n$ if $L' = Z(L)$ and $\dim L' = n$.
\end{definition}

A freest generalized Heisenberg Lie algebra is  a $d$-generated (minimally generated by $d$ elements) generalized Heisenberg Lie algebra of rank $\frac{1}{2}d(d-1)$ for some $d \geq 2$.  We shall denote it by $L_d$. Notice that $L_d$ has the following presentation:
\[\gen{x_1, \ldots, x_d, \  y_{ij}  \mid  \ [x_i, x_j ] = y_{ij}, 1 \leq i < j \leq d, \ \text{class} \ 2},\]
and $\dim L_d = \frac{1}{2}d(d+1).$

\begin{theorem}\label{thm}
Let $L_d$ be the freest generalized Heisenberg Lie algebra  of rank $\frac{1}{2}d(d-1)$. Then $B_0(L_d) = 0.$
\end{theorem}

\begin{proof}
Let $f: L_d \times L_d \mapsto \Omega$ be a 2-cocycle such that $\bar{f} \in B_0(L_d)$. Let $B = \{x_1, \ldots, x_d, [x_i, x_j] \mid \ 1 \leq i < j \leq d\}$. Define a map $\mu:B \mapsto \Omega$ as follows:
 \[\mu(x_i) = 0 \ \text{for} \ 1 \leq i \leq d,\]
 \[\mu([x_i,x_j]) = -f(x_i,x_j) \ \text{for} \ 1 \leq i < j \leq d.\]
Since $B$ is basis for $L_d$ we can extend this map linearly to $L_d$ and call it $\sigma$ so that $\sigma$ is a linear transformation. Let $x, y \in L_d$. Since $B$ is a basis we can write $x$ and $y$ as 
\[x = \sum_{i=1}^{d} \alpha_ix_i +\sum_{1 \leq i < j \leq d}^{}\alpha_{ij}[x_i,x_j],\]
\[y = \sum_{k=1}^{d} \beta_kx_k +\sum_{1 \leq k < l \leq d}^{}\beta_{kl}[x_k,x_l],\]
for some $\alpha_i, \beta_i, \alpha_{ij}, \beta_{ij} \in \Omega.$
Now using the bilinearity of $f$ and the fact that $f(a,b) = 0$ whenever $[a,b] = 0$, we get that
\[f(x,y) = \sum_{i=1}^{d}\sum_{k=1}^{d}\alpha_i\beta_kf(x_i,x_k).\]
%\[f(x,y) = \sum_{i=1}^{d}\sum_{k=1}^{d}\alpha_i\beta_kf(x_i,x_k) + \sum_{i=1}^{d}\sum_{1 \leq k < l \leq d}^{}\alpha_i\beta_{kl}f(x_i,[x_k,x_l]) + \sum_{k=1}^{d}\sum_{1 \leq i < j \leq d}^{}\alpha_{ij}\beta_kf([]x_i,x_j] x_k)

Also, using the bilinearity of Lie bracket, linearity of $\sigma$ and the fact that $f(l_1, [l_2, l_3]) = 0$ for all $l_1, l_2, l_3 \in L_d$, we get that  
\[\sigma([x,y]) = -\sum_{i=1}^{d}\sum_{k=1}^{d}\alpha_i\beta_kf(x_i,x_k).\]
Therefore $f(x,y) = -\sigma([x,y]).$ Thus $f$ is a coboundry and $\bar{f} = 0$. This completes the proof.

\end{proof}

Rostami et al. \cite{Niroomand} proved that the Bogomolov multiplier of a Heisenberg Lie algebra is trivial. We prove this fact as a Corollary of Theorem \ref{thm}.\\

A Heisenberg Lie algebra of dimension $2n+1, n >0,$ is given by the following presentation:
\[\gen{x_1, \ldots, x_{2n}, \  v  \mid  \ [x_{2i-1}, x_{2i} ] = v, [x_j,x_k] = 0,  1 \leq i \leq n, (j,k) \neq (2i-1, 2i)}.\]

\begin{corollary}
Let $H_{2n+1}$ be the Heisenberg Lie algebra of dimension $2n+1.$ Then $B_0(H_{2n+1}) = 0.$
\end{corollary}

\begin{proof}
 Let $L_{2n}$ be the generalized Heisenberg Lie algebra of rank $n(2n-1)$ and $M$ be its ideal generated by $[x_{2r-1}, x_{2r}]-[x_{2s-1}, x_{2s}], 1 \leq r < s \leq n$ and $[x_t, x_u], 1 \leq t < u \leq 2n, (t, u) \neq (2i-1,2i)$ for any $i \leq n.$ Then it is easy to see that $L_{2n}/M$ is isomorphic to $H_{2n+1}.$ Since $B_0(L_{2n}) = 0$ by Theorem \ref{thm}, it follows from Theorem \ref{thm0} that 
 \[B_0(L_{2n}/M) \cong \frac{M \cap L_{2n}'}{\gen{K(L_{2n})\cap M}}.\]
 Since for $1 \leq r < s \leq n$,
 \[[x_{2r-1}-x_{2s-1}, x_{2r}+x_{2s}] = [x_{2r-1}, x_{2r}]- [x_{2s-1}, x_{2s}]+[x_{2r}, x_{2s-1}] +[x_{2r-1}, x_{2s}],\]
 it follows that $M \cap L_{2n}' = \gen{K(L_{2n})\cap M}.$
 Thus $B_0(L_{2n}/M) = 0$ so that $B_0(H_{2n+1}) = 0.$
 
\end{proof}

We now proceed to prove Theorem \ref{thm1}.\\

{\bf Proof of Theorem \ref{thm1}:}
Let $d$ be a natural number greater than $4n$ and $L_d$ be the $d$-generated freest generalized Heisenberg Lie algebra generated by $x_1, x_2, \ldots x_d$. Let $M$ be the ideal generated by $[x_1, x_2]+[x_3, x_4], [x_5, x_6]+[x_7, x_8], \ldots [x_{4n-3}, x_{4n-2}]+[x_{4n-1}, x_{4n}].$ Since $B_0(L_d) = 0$, it follows from Theorem \ref{thm0} that 
\[\dim B_0(L_d/M) = \dim \frac{L_d' \cap M}{\gen{K(L_d) \cap M}}.\]
 Next we prove that $K(L_d) \cap M = \{0\}$. For this let $l \in K(L_d) \cap M.$ Since $l \in M$ there exists $\gamma_k's$ such that
\[l=\sum_{k=0}^{n-1}\gamma_{k+1}\Big([x_{4k+1}, x_{4k+2}]+[x_{4k+3}, x_{4k+4}]\Big).\]
 Also there exist $\alpha_i$'s and $\beta_j$'s such that
\[l=\bigg{[}\sum_{i=1}^{d}\alpha_ix_i,\sum_{j=1}^{d}\beta_jx_j\bigg{]}\]
because $l \in K(L_d).$ Hence
\[\sum_{k=0}^{n-1}\gamma_{k+1}\Big([x_{4k+1}, x_{4k+2}]+[x_{4k+3}, x_{4k+4}]\Big)=\sum_{i=1}^{d-1}\sum_{j=i+1}^{d}(\alpha_i\beta_j-\alpha_j\beta_i)[x_i, x_j].\]
It follows that 
\begin{equation}\label{eq0}
\alpha_i\beta_j-\alpha_j\beta_i = 0  \ \ \text{when} \ \ (i,j) \neq (2k+1, 2k+2) \ \text{for any} \ k= 0,1,\ldots n-1,
\end{equation}
and 
\[\gamma_{k+1} = \alpha_{4k+1}\beta_{4k+2}-\alpha_{4k+2}\beta_{4k+1} = \alpha_{4k+3}\beta_{4k+4}-\alpha_{4k+4}\beta_{4k+3}\]
 \text{for any} \  $k = 0,1,\ldots n-1$. From Equation \ref{eq0} we have 
\begin{equation}\label{eq1}
\alpha_{4k+1}\beta_{4k+3}-\alpha_{4k+3}\beta_{4k+1}=0 
\end{equation}
\begin{equation}\label{eq2}
\alpha_{4k+1}\beta_{4k+4}-\alpha_{4k+4}\beta_{4k+1}=0
\end{equation}
\begin{equation}\label{eq3}
\alpha_{4k+2}\beta_{4k+3}-\alpha_{4k+3}\beta_{4k+2}=0
\end{equation}
\begin{equation}\label{eq4}\alpha_{4k+2}\beta_{4k+4}-\alpha_{4k+4}\beta_{4k+2}=0,\end{equation}
for $k=0,1,\ldots n-1.$

Suppose $\alpha_{4k+1} = \beta_{4k+1} = 0.$ Then $\gamma_{k+1} =0.$ Assume then that $\alpha_{4k+1} = 0$ but  $\beta_{4k+1} \neq 0.$ From Equations \ref{eq1} and \ref{eq2}, $\alpha_{4k+3} = \alpha_{4k+1} = 0$. As a result, $\gamma_{k+1} = 0$ in this case as well. Thus we have shown that if $\alpha_{4k+1} = 0$, then $\gamma_{k+1} = 0$. Similarly, if either of $\alpha_{4k+2}, \alpha_{4k+3}, \alpha_{4k+4}, \beta_{4k+1}, \beta_{4k+2}, \beta_{4k+3}, \beta_{4k+4}$ is zero, then $\gamma_{k+1} = 0.$ Hence, we can now assume that neither of $\alpha_{4k+i}, \beta_{4k+i}$ is zero for $i = 1, 2, 3, 4.$ From Equations \ref{eq1} and \ref{eq2} we can deduce that $\beta_{4k+3}/\beta_{4k+4} = \alpha_{4k+3}/\alpha_{4k+4}.$ Hence $\gamma_{k+1}=0$ so that $l=0$. It follows that  $K(L_d) \cap M = \{0\}$.

Also, $M \leq L_d'$. Hence $\dim B_0(L_d/M) = \dim (M) = n.$ By Corollary \ref{cor1}  $\dim B(L_d/M) = n.$ Taking $L$ to be $L_d/M$ completes the proof.\\

{\bf Proof of Theorem \ref{thm2}:} Let $(\theta, \phi)$  be an isoclinism between $L$ and $M$, i.e., $\theta: \frac{L}{Z(L)} \mapsto \frac{M}{Z(M)}$ and $\phi:\gamma_2(L) \mapsto \gamma_2(M)$ be isomorphisms and whenever $\theta(l_iZ(L)) = m_iZ(M)$ for $i =1,2$, we have $\phi([l_1, l_2]) = [m_1, m_2].$ Let $\bar{f} \in B_0(L)$ where $f:L \times L \mapsto \Omega$ be a cocycle.
Define $c_f: M \times M \mapsto \Omega$ by $c_f(m_1, m_2) = f(l_1, l_2)$, where $l_1$ and $l_2$ are given by $\theta^{-1}(m_i+Z(M))= l_i+Z(L)$ for $i= 1,2.$ The rest of the proof follows from the following lemmas:

\begin{lemma}\label{lem1}
Let $c_f$ be the map defined above. Then 
\begin{itemize}
    \item[(i)] $c_f$ is well defined.
    \item[(ii)] $c_f$ is a 2 cocycle.
    \item[(iii)] $\overline{c_f} \in B_0(M).$
\end{itemize}
\end{lemma}

\begin{proof}
Since $f$ is bilinear and $f(k,l) = 0$ whenever $[k,l] = 0$. It follows that $f(l_1+z_1, l_2+z_2) = f(l_1,l_2)$ for every $z_1, z_2 \in Z(L)$. This shows that $c_f$ is well-defined.

To see that $c_f$ is a 2-cocycle, let $m_1, m_2, m_3 \in M$ and let $\theta^{-1}(m_i+Z(M))= l_i+Z(L)$ for $i= 1,2,3.$ It is obvious that $\theta^{-1}(m_1+m_2+Z(M))= l_1+l_2+Z(L).$ Therefore $c_f(m_1+m_2, m_3) = f(l_1+l_2, l_3)$ which equals $f(l_1, l_3)+ f(l_2, l_3)$ that is equal to  $c_f(m_1, m_3)+c_f(m_2, m_3).$ Similarly $c_f(m_1, m_2+m_3) = c_f(m_1, m_2)+c_f(m_1, m_3)$. Thus $c_f$ is bilinear. Also, it is easy to see that $c_f$ is alternating because $f$ is alternating. Next, For $i,j,k \in \{1,2,3\}$ note that $c_f([m_i,m_j], m_k) = f([l_i, l_j],l_k)$ because 
\begin{equation*}
\begin{split}\theta^{-1}\Big([m_i,m_j]+Z(M)\Big) &= \theta^{-1}\Big(\Big[m_i+Z(M), m_j+Z(M)\Big]\Big)\\
&= \Big[\theta^{-1}\Big(m_i+Z(M)\Big), \theta^{-1}\Big(m_j+Z(M)\Big)\Big]\\
&= \Big[l_i+Z(L),l_j+Z(L)\Big]\\
&= [l_i,l_j]+Z(L).
\end{split}
\end{equation*}
\textit{}t follows that $c_f$ is a 2-cocycle, since $f$ is a 2-cocycle.

To see that $\overline{c_f} \in B_0(M)$, suppose that $[m_1, m_2]= 0$. But then $[l_1,l_2] =0$ because $\phi([l_1, l_2]) = [m_1,m_2].$ Since $\overline{f} \in B_0(L)$ it follows that  $f(l_1,l_2) = 0.$ Hence $c_f(m_1,m_2) =0.$ 

\begin{lemma}
The map $\eta: B_0(L) \mapsto B_0(M)$ defined by $\eta(\overline{f}) = \overline{c_f}$ is an isomorphism.
\end{lemma}

\begin{proof}
 We begin by ensuring that the map is well-defined. To verify this consider $\sigma: L \times L \mapsto \Omega$ to be a coboundary. Then 
 \[c_{f+\sigma}(m_1, m_2) = (f+\sigma)(l_1,l_2) = f(l_1, l_2)+\sigma(l_1,l_2) = c_f(m_1, m_2)+c_{\sigma}(m_1,m_2).\]
 Thus we have,  $c_{f+\sigma}= c_f+c_{\sigma}.$ Notice that $c_{\sigma}$ is a coboundary because $\sigma$ is coboundary. Therefore $\overline{c_f} = \overline{c_{f+\sigma}}$, i.e., $\eta(\overline{f})= \eta(\overline{f+\sigma}).$ This proves that $\eta$ is well-defined. 
 
 In a similar fashion one can see that $c_{f_1+f_2} = c_{f_1}+c_{f_2}$ and $c_{\alpha f_1} = \alpha c_{f_1}$ for each $\alpha \in \Omega$ and each cocycles $f_1$, $f_2$ from $L \times L$ to $\Omega$. So that $\eta(\overline{f_1+f_2})= \eta(\overline{f_1})+\eta(\overline{f_2})$ and $\eta(\alpha\overline{f_1}) = \alpha\eta(\overline{f_1}).$ Thus $\eta$ is a linear map.
 
 Finally, in order to see that $\eta$ is a bijection, we define another map $\chi: B_0(M) \mapsto B_0(L)$ in the same way as $\eta$ is defined from $B_0(L)$ to $B_0(M).$ Then it is easy to see that $\eta\chi$ and $\chi\eta$ both are identity maps and thus $\eta$ is a bijection. This completes the proof.
\end{proof}

\end{proof}

%\begin{remark}
%The analogous result, to the Theorem \ref{thm2}, for finite groups was proved by Moravec \cite{Moravec1}. A different proof can be given using the same ideas as used here in the proof of Theorem \ref{thm2}. However one needs to use the elementary but non-trivial Lemma
%\end{remark}

\vspace{.3cm}

%{\bf Acknowledgements:} 

\end{document}